\documentclass{amsart}

\usepackage{amsmath}
\usepackage{amssymb}

\usepackage{amssymb}
\usepackage{enumerate}
\usepackage{color}
\def\CC{\mathbb C}
\def\CO{\mathcal O}
\def\conv{\operatorname{conv}}
\def\phi{\varphi}
\def\Ree{\operatorname{Re}}
\def\RR{\mathbb R}
\def\NN{\mathbb N}
\def\PP{\mathbb P}
\def\DD{\mathbb D}
\def\eps{\varepsilon}
\def\too{\longrightarrow}
\def\Cal{\mathcal}
\def\wdht{\widehat}
\def\wdtl{\widetilde}
\def\rank{rank}
\makeatletter
\def\th@mytheorem{%
  \let\thm@indent\noindent
  \thm@headfont{\bfseries}
    \itshape
}
\def\th@myremark{%
  \let\thm@indent\noindent
  \thm@headfont{\bfseries}
}
\makeatother
\theoremstyle{mytheorem}
\newtheorem{Theorem}{Theorem}
\newtheorem{nono-theorem}{Theorem}[]

\newtheorem*{Theorem*}{Theorem}
\theoremstyle{myremark}
\newtheorem{Remark}[Theorem]{Remark}
\newtheorem{Lemma}{Lemma}
\begin{document}

\title[A note on partial polynomial functions] {A note on partial polynomial functions,\\ in memory of Marek Jarnicki}

\author[P.~Pflug]{Peter Pflug}
\address{Carl von Ossietzky Universit\"at Oldenburg, Institut f\"ur Mathematik,
Postfach 2503, D-26111 Oldenburg, Germany}
\email{peter.pflug@uol.de}

\begin{abstract}
We present an extension theorem for a separately holomorphic function which is polynomial/rational in some variables.
\end{abstract}



\maketitle

 Motivated by the recent paper \cite{Ivan2024} we discuss the following generalizations of their results, see Theorem \ref{eins} and Theorem \ref{zwei}. The main ideas are taken from \cite{Ivan2024} enriched by results on cross domains in \cite{JarPfl2011}.

 First, let us recall some notation and results from \cite{JarPfl2011}. Given domains $D\subset\CC^m$, $G\subset\CC^n$ and subsets $A\subset D$, $B\subset G$, both locally pluriregular in $D$, resp. $G$, then one defines the associated cross
 $$
 \bold X:=\bold X(A,D;B,G):=(A\times G)\cup(D\times B)
 $$
 and its envelope
 $$
 \wdht{\bold X}:=\{(z,w)\in D\times G:h^\ast_{A,D}(z)+h^\ast_{B,G}(w)<1\},
 $$
 where $h_{A,D}$ (resp. $h_{B,G}$) is the associated relative extremal function (for its definition see \cite{JarPfl2011}) and $h_{B,G}^\ast$ its upper semicontinuous regularization.

 Note that $\bold X$ and $\wdht{\bold X}$ are both connected and $\bold X\subset\wdht{\bold X}$ (since the sets $A$ and $B$ are assumed to be locally pluriregular).

 Moreover, a function $f:\bold X\too\CC$ is called to be {\it separately holomorphic} on the cross $\bold X$, if for all $a\in A$ (resp. for all $b\in B$) the function $f(a,\cdot)$ (resp. $f(\cdot,b)$) is holomorphic on $G$ (resp on $D$); we will shortly  write $f\in\CO_s(\bold X)$.

 Let us also recall the cross theorem (see Theorem 5.4.1 in \cite{JarPfl2011}), which is the basic result for the proofs in this note.
 \vskip 0.2cm

\begin{Theorem*} If $f\in\CO_s(\bold X)$, then there exists an $\wdht f\in \CO(\wdht{\bold X})$\footnote{For a domain $S\subset\CC^k$  the set of all holomorphic functions on $S$ is as usually denoted by $\CO(S)$.} with $\wdht f|_{\bold X}=f$.
\end{Theorem*}
\vskip 0.2cm
In addition, I add two  observations we will use later:
\begin{Remark}\label{rem1}
$\alpha$) if $M$ is a locally pluriregular subset of the domain $S\subset\CC^k$ and $U\subset \CC^k$ open with $U\cap M\neq\varnothing$, then $U\cap M$ is also locally pluriregular (use just the definition);

$\beta$) if $M\subset S\subset\CC^k$ is locally pluriregular and $f\in\CO(S)$ with $f|_M=0$, then $f=0$ on $S$.\footnote{
Otherwise, $M\subset\{z\in S:f(z)=0\}=:M'\subsetneqq S$, i.e. $M'$ is an analytic subset of $S$, different from $S$, and so it is pluripolar. Therefore, $M$ is pluripolar. Using now \cite{JarPfl2011}, Corollary 3.2.12, gives $h_{M,S}^\ast\equiv 1$ on $S$ which contradicts the fact that $M$ is assumed to be locally pluriregular (i.e. $h_{M,S}^\ast|_M\equiv 0$).}.
\end{Remark}

Finally we state some simple facts.

\begin{Lemma}\label{lem1} Fix an $f\in\CO_s(\wdht X)$.  Let $B':=\{w\in B: f(\cdot,w)\not\equiv 0\}$. Then

1) $B'$ is open as a subset of $B$;

2) if $0\in A$ and  if $B\setminus B'$ contains a non-empty relatively open subset $U$ with $f(\cdot,w)\equiv 0$ on $D$ for all $w\in U$, then $B'=\varnothing$, i.e. $f(\cdot,w)\equiv 0$ for all $w\in B$.

3) if $0\in A$, then $B'$ is a dense subset of $B$.
\end{Lemma}

\begin{proof}
1)  Take a $w_0\in B'$ and a $z_0\in D$ such that $f(z_0,w_0)\neq 0$. Since $(z_0,w_0)\in\bold X\subset\wdht{\bold X}$ we have $\wdht f(z_0,w_0)=f(z_0,w_0)\neq 0$. Then there exists an open neighborhood $U\subset G$ of $w_0$ such that $\wdht f(z_0,w)\neq 0$, $w\in B\cap U$, i.e. $B\cap U\subset B'$.

2) By assumption, $f(\cdot,w)\equiv 0$ on $D$ in case that $w\in U$. Let $\wdht f$ be the holomorphic extension of $f$ to $\wdht{\bold X}$ according the cross theorem. Recall that $\{0\}\times G\subset\wdht{\bold X}$. Then $\wdht f$ can be written as its Hartogs series with center $0\in\CC^m$, i.e. $\wdht f(z,w)=\sum_{\alpha\in\NN_0^n}c_\alpha(w)z^\alpha$, where the $c_\alpha$ are holomorphic functions on $G$. Note that $c_\alpha(w)=0$ whenever $w\in U$. Hence by Remark \ref{rem1}  it follows that $c_\alpha\equiv 0$ on $G$; in particular, $f(\cdot,w)\equiv 0$ for all $w\in B$.

3) This claim is a direct consequence of the statements before.

\end{proof}

 With this information at hand we formulate our first result.

\begin{Theorem}\label{eins}
Let $D\subset\CC^m$, $G\subset\CC^n$ be domains and  $A\subset D$, resp. $B\subset G$,  locally pluriregular subsets of $D$, resp. of $G$. Moreover, $B$ is assumed to be closed as a subset of $\CC^n$.
Put $\bold X:=(A\times G)\cup (D\times B)\subset\CC^N$ and let $f:\bold X\too \CC$ be a given function with the following properties:

(a) for any point $a\in A$: $f(a,\cdot)\in\CO(G)$,

(b) for any $b\in B$: $f(\cdot,b)$ is the restriction of a polynomial $P_b\in\CC[z_1,\dots,z_m]$ \footnote{$P_b$ is allowed to be the trivial polynomial $P_b(z)\equiv 0$.}.

Then there exists a function $\wdht f\in\CO(\CC^m\times G)$ with $\wdht f|_{\bold X}=f$ and $\wdht f\in\CO(G)[z_1,\dots,z_m]$.

In particular, any $f(\cdot,b)$, $b\in B$, is the restriction of a polynomial whose coefficients are holomorphic on $G$.

\end{Theorem}

Note that the $f$ in the theorem belongs to $\CO_s(\bold X)$.

\begin{proof} First, without loss of generality we may assume that $0\in A$ and that there is at least one point $w'\in B$ such that $f(\cdot,w')\not\equiv  0$ on $D$ (otherwise the statement in the theorem becomes trivial).

Note  that the function $\wdtl f:(A\times G)\cup (\CC^m\times B)\too \CC$ defined as
$$
\wdtl f(z,w):=\begin{cases} f(z,w), &\text{ if } (z,w)\in A\times G \\
P_w(z), &\text{ if } (z,w)\in\CC^m\times B
\end{cases}
$$
belongs to $\CO_s(\bold Y)$ with the cross $\bold Y:=(\CC^m\times B)\cup(A\times G)$, i.e., $\wdtl f(a,\cdot)\in\CO(G)$ for all $a\in A$ and $\wdtl f(\cdot,b)\in\CO(\CC^m)$ for all $b\in B$.

Therefore, using the general cross theorem (see Theorem 5.4.1 in \cite{JarPfl2011}) one concludes that there exists $\wdht f\in\CO(\wdht{\bold Y})$ with $\wdht f=\wdtl f$ on $\bold Y$,
where $$
\wdht{\bold Y}=\{(z,w)\in \CC^m\times G:h^\ast_{A,\CC^m}(z)+h^\ast_{B,G}(w)<1\}.
$$
In particular, $\wdht f=f$ on $\bold X$,

Observe that the relative extremal function $h^\ast_{A,\CC^m}=0$. Therefore, $\wdht f\in\CO(\CC^m\times G)$.

For the non-trivial polynomial $P_w$ denote by $k(w)\in\NN_0$\footnote{$\NN_0:=\NN\cup\{0\}$.}, $w\in B$, its degree. Moreover, let $B':=\{w\in B:f(\cdot,w)\not\equiv 0\}$.

Put $M_k:=\{w\in B:k(w)= k\}$. Since the closed set $B$ is the union of these countable many sets $M_k$ and $B\setminus B'$, it follows applying Baire's theorem and the second part of Lemma \ref{lem1}, that there is a $k_0\in\NN$ such that $\overline {M_{k_0}}$ contains an non-empty relatively open subset  $U\cap B$, $U\subset G$ open.

Recall that $\wdht f$ is holomorphic on the Hartogs domain $\CC^m\times G$. Therefore, $\wdht f$ can be written as its Hartogs series $\wdht f(z,w)=\sum_{\alpha\in\NN_0^m}c_\alpha(w) z^\alpha$ with $c_\alpha\in\CO(G)$. Then $c_\alpha(w)=0$ for all $|\alpha|>k_0$ and $w\in U\cap B$ because of $U\cap B
\subset\overline M_{k_0}$ and $c_\alpha$ is continuous on $U$.

Note that the set $B\cap U$ is locally pluriregular and the holomorphic functions $c_\alpha$, $|\alpha|>k_0$, vanishes on $U\cap B$. Hence, these $c_\alpha$ are identical zero on the whole of $G$..

Then
$\wdht f(z,w)=\sum_{|\alpha|\leq k_0}c_\alpha(w)z^\alpha$ on $\CC^m\times G$, which completes the proof of the theorem.
\end{proof}

\begin{Remark}
1) In case that $B$ contains a closed ball $K$, then $K$ is locally pluriregular (use the Oka Theorem \cite{Vlad}). Hence Theorem \ref{eins} is true and so Corollary 1 in \cite{Ivan2024} follows directly from Theorem \ref{eins}.

2) It is unclear to us whether Theorem \ref{eins} remains true if the additional assumption that $B$ is closed in $\CC^n$ is cancelled.
\end{Remark}

Similar as the main result in \cite{Ivan2024} even more is true  \footnote{See also \cite{Sic1962}, where Kronecker's criterion for the rationality of a function of one complex variable was already applied.}.

\begin{Theorem}\label{zwei}
Let $D\subset\CC^m$, $G\subset\CC^n$ be domains, $G$ a domain of holomorphy, and let $A\subset D$, resp. $B\subset G$, locally pluriregular. Moreover assume that $B$ is compact in $\CC^n$.
Put $$
\bold X:=(A\times G)\cup(D\times B)\subset \CC^m\times\CC^n=\CC^N.
$$
Let $f\in\CO_s(\bold X)$ satisfying in addition the following property:

 for any $w\in B$ there exist polynomials $P_w, Q_w\in\CC[z_1,\dots,z_m]$, $Q_w\neq 0$, such that $f(\cdot,w)Q_w=P_w$ on $D$,

Then there exist a domain of holomorphy $G'$ with
$B\subset G'\subset G$, a proper analytic subset $F$ of $G'$, and polynomials $P, Q\in\CO(G')[z_1,\dots,z_m]$,
such that $fQ=P$ on $\bold X$. Moreover, if $w\in B\setminus F$, then $Q(\cdot,w)$ is not the trivial polynomial.
In particular, the function $f(\cdot,w)$ is a rational function for all $w\in B\setminus F$ with holomorphic coefficients on $G'$.

Moreover, if $m=1$, then the theorem is even true if $B$ is only assumed to be closed in $\CC^n$, $G$ not necessarily a domain of holomorphy, and $G'=G$.
 \end{Theorem}

Note that for $m=1$ the result \footnote{Observe that the result there has to be modified taking into account of an exceptional analytic subset in the parameter space.} in \cite{Ivan2024} is an immediate consequence of Theorem \ref{zwei}. Namely, take $A=D$ and $B$ a closed ball inside of $G$. Then $\bold X=(D\times G)\cup (D\times B)$, $G'= G$, and therefore $fQ=P$ with $P, Q\in\CO(G)[z_1]$ with the properties mentioned in Theorem \ref{zwei}.

However, it remains unclear how to get from Theorem \ref{zwei} directly the result in \cite{Ivan2024} in case $m>1$.

\begin{Remark}
a) At the moment it remains open what happens if $B$ is only assumed to be a locally pluriregular subset of $G$.

b) It is also open whether the neighborhood $G'$ of $B$ can be chosen to be the whole $G$.

c) It is also not clear whether the assumption that $G$ is a domain of holomorphy can be cancelled.

\end{Remark}

\begin{proof}
First some general remarks:

Without loss of generality we may assume that $0\in A$ and that there exists a point $\wdht w\in B$ such that $f(\cdot,\wdht w)\not\equiv 0$ (otherwise take $Q(z,w)=1$ and $P(z,w)=0$ whenever $w\in G$ and $z\in\CC^m$).


Since $f\in\CO_s(\bold X)$, there exists a function $\wdht f\in\CO(\wdht{\bold X})$, where
$$
\wdht{\bold X}=\{(z,w)\in D\times G: h^\ast_{A,D}(z)+h^\ast_{B,G}(w)<1\},
$$
such that $\wdht f|_{\bold X}=f$.

Note that $h^\ast_{A,D}|_A=0$. Thus $\{0\}\times G\subset\wdht{\bold X}$. Therefore, there is a continuous positive function $\eps:G\too(0,\infty)$ such that
$$
H:=\{(z,w)\in \CC^m\times G: \|z\|<\eps(w)\}\subset\wdht{\bold X}.
$$
Here $\|\cdot\|$ denotes the maximum-norm in $\CC^m$.

Hence, we may write $\wdht f$ as its Hartogs series on $H$, i.e. $\wdht f(z,w)=\sum_{\alpha}\wdht c_\alpha(w)z^\alpha$ with $\wdht c_\alpha\in\CO(G)$. Note that not all the $\wdht c_\alpha(\wdht w)$ vanish since otherwise (use the identity theorem) $f(\cdot,\wdht w)$ is identically zero on $D$.

Moreover, we may assume that for any $w\in B$ the polynomials $P_w, Q_w$ are relative prime. Note that $P_w/Q_w$ coincides with $f(\cdot,w)$ on $D$, i.e. it is holomorphic on $D$. Therefore, using \cite{RudSto}, $Q_w$ has no zeros in $D$.


\vskip 0.3cm

For the next step we assume that $m=1$:

Observe first that in case there is an $s_0\in\NN$ such that for all $s\geq s_0$ and all $w\in B$ one has $\wdht c_s(w):=\frac{1}{s!}\frac{\partial^s f}{\partial^sz_1}(0,w)=0$. Then
$f(\cdot,w)\in\CO(G)[z_1]$ or more precisely $f(z_1,w)=\sum_{s=0}^{s_0}\wdht c_s(w)z_1^s$ where the coefficients $\wdht c_s$ are holomorphic on $G$.


Next we follow the argument given in \cite{Ivan2024}. For simplicity, from now on we will write $z$ instead of $z_1$.

Let $P_w(z)=\sum_{j=0}^{k(w)-1}a_j(w)z^j$ and $Q_w(z)=\sum_{j=0}^{k(w)}b_j(w)z^j$, where $w\in B$ and $k(w)\in\NN$. Note that $k(w)$ (resp. $k(w)-1$) is not necessarily the degree of $Q_w$ (resp. $P_w$).
Observe that for a fixed $w\in B$ not all $b_j(w)$ are zero.


Recall that $f(\cdot,w)$, $w\in B$, is holomorphic on $D$ and that we assume that $0\in A$. Take $r>0$ such that $\overline\DD(0,r)\subset D$, where $\DD(0,r)$ denotes the open disc with center $0$ and radius $r$. Then $f(z,w)$ can be written on $\overline\DD(0,r)\times B$ as $f(z,w)=\sum_{j=0}^\infty c_j(w)z^j$. Recall that in addition
$\wdht f(z,w)=\sum_{j=0}^\infty \wdht c_j(w)z^j$ on $\overline\DD(0,\eps(w)/2)\times G$, where the $\wdht c_j$'s are holomorphic on $G$. Since $\wdht f(\cdot,w)=f(\cdot,w)$ for $w\in B$, it follows that $\wdht c_j(w)=c_j(w)$ for $w\in B$. Hence, $c_j$ is the restriction to $B$ of the holomorphic function $\wdht c_j\in\CO(G)$.

Recall from above that for a $w\in B$ such that $f(\cdot,w)\not\equiv 0$ on $D$ one knows that not all $c_j(w)$ vanish.

Let now $w\in B$. Then, comparing the coefficients in the equation $f(\cdot,w)Q_w=P_w$, we get the following equations with $k:=k(w)$:

$$ \begin{aligned}
a_0(w)&=c_0(w)b_0(w)\\
a_1(w)&=c_1(w)b_0(w)+c_0(w)b_1(w)\\
&\text{etc}\\
a_{k-1}(w)&=c_{k-1}(w)b_0(w)+c_{k-2}(w)b_1(w)+\cdot +c_0(w)b_{k-1}(w).
\end{aligned}
$$
Moreover,
$$
\begin{aligned}
0&=c_k(w) b_0(w)+c_{k-1}(w)b_1(w)+\cdots + c_0(w)b_k(w) \\
0&=c_{k+1}(w)b_0(w)+c_k(w)b_1(w)+\cdots + c_1(w)b_{k}(w) \\
&\text{etc}\\
0&=c_{k+s}(w)b_0(w)+c_{k+s-1}(w)b_1(w)+\cdots + c_{s}(w)b_k(w),
\end{aligned}
$$
where $s\in\NN_0$.

Recall the notion of the Hankel matrix for $w\in G$, i.e.
$$
H_{s,m}(w):=\left(\begin{matrix} \wdht c_s(w) & \wdht c_{s+1}(w) & \cdots & \wdht c_{s+m}(w)\\
\wdht c_{s+1}(w) & \wdht c_{s+2}(w) & \cdots & \wdht c_{s+1+m}(w)\\
\cdots\\
\wdht c_{s+m}(w) & \wdht c_{s+m+1}(w) & \cdots & \wdht c_{s+2m}(w)\end{matrix}\right),\quad  s, m\in\NN_0.
$$

Then, since the coefficients of $Q_w$ are not identically zero, one gets
$$
A_{s,k(w)}(w):=\det H_{s,k(w)}(w)=0, \quad w\in B, s\geq 0.
$$
Note that $k(w)$ depends on the point $w\in B$.

Now, for $\ell\in \NN$ put
$$
M_\ell:=\{w\in B: k(w)=\ell\}.
$$

Then $B$ is the countable union of the $M_\ell$'s. Hence applying Baire leads to an $\ell_0\in\NN$  such that $\overline{M_{\ell_0}}$  contains a non-empty relatively open subset of $B$, i.e. there is an open subset $U\subset\CC^n$ with $\varnothing\neq U\cap B\subset \overline{M_{\ell_0}}$.

 If $w\in M_{\ell_0}$, then $A_{s,\ell_0}(w)=0$. Note that $A_{s,\ell_0}$ is a holomorphic function on $G$. Therefore, $A_{s,\ell_0}(w)=0$ for all $w\in U\cap B$. Finally using that $U\cap B$ is again locally pluriregular yields that $A_{s,\ell_0}=0$ on the whole of $G$.

\emph {To summarize:} there is  an $\ell_0\in\NN$  such that for all $s\in\NN_0$ we have $A_{s,\ell_0}\equiv 0$ on $G$.

The rest of the proof for $m=1$ is based on a result by Kronecker (\cite{Kron1881}), see also Borel (\cite{Bor1894}), giving a criterion for the rationality  of a holomorphic function in one complex variable.

Recall that all the matrices $H_{s,\ell_0}(w)$, $w\in G$ and $s\in \NN_0$, are singular and that $\ell_0\in\NN$. Now choose $r_0\in\NN$ to be the smallest number such that all the matrices $H_{s,r_0}(w)$, $s\in\NN_0, w\in G$, are singular. Note that $r_0\leq \ell_0$.

First, let $r_0=1$:

If $w\in G$ with  $\wdht c_0(w)=0$, then $\wdht c_s(w)=0$ for all $s\geq 0$. Therefore, there exists a $w'\in B$ such that $\wdht c_0(w')\neq 0$, otherwise $f(\cdot,w)=0$ on $D$ for all $w\in B$.\footnote{Note that if $w\in B$ with $f(\cdot,w)\not\equiv 0$ on $D$, then $c_0(w)\neq 0$.} Let $F:=\{w\in G: \wdht c_0(w)=0\}$ which is a proper analytic subset of $G$. For $w\in G\setminus F$ put $\wdtl b_1(w):=\wdht c_1(w)/\wdht c_0(w)$ and $\wdtl b_0(w):=-1$. Then $\wdht c_0(w)\wdtl b_1(w)+\wdht c_1(w)\wdtl b_0(w)=0$ on $G\setminus F$. Because of $r_0=1$ it is easy to see that the vectors $(c_s(w),c_{s+1}(w)), (c_{s+1}(w),c_{s+2}(w))$ are linear dependent whenever $s\geq 0$. Consequently, the following equations
$$
\wdht c_s b_1+\wdht c_{s+1}b_0=0,\quad s\geq 0,
$$
hold on $G\setminus F$, where $b_j(w):=\wdtl b_j(w)\wdht c_0(w)$. Hence these equations are true on the whole of $G$. Note, that the $b_j$ are holomorphic on $G$. Finally, let $a_0(w):=\wdht c_0(w)b_0(w)$, $w\in G$.

Hence, $$
\wdht f(z,w) \Big( b_1(w)z+b_0(w)\Big)=a_0(w)
$$ as long as $w\in G$ and $\|z\|<\eps(w)$; in particular, for $(z,w)\in\bold X$.

Put $Q(z,w):=b_0(w)+b_1(w)z$ and $P(z,w):=a_0(w)$. 

Note that the polynomial $Q(z,w)$ is non-trivial, if $w\in B\setminus F$.
\vskip 0.2cm

Finally we assume $r_0\geq 2$:

Assume that for an arbitrary point $w\in G$  that $A_{0,r_0-1}(w)=0$. Using the following identity (see \cite{Had1892}, \cite{Bor1894})
$$
A_{s,r_0-1}(w)A_{s+2,r_0-1}(w)-A_{s+1,r_0-1}^2(w)=A_{s,r_0}(w)A_{s+2,r_0-2}(w) , \quad s= 0,\dots \quad (\ast)
$$
yields $A_{1,r_0-1}(w)=0$. Applying ($\ast$) again and again we get that $A_{s,r_0-1}(w)=0$ for all $s\geq 0$. Since $w$ was arbitrarily chosen this contradicts the minimality assumption for $r_0$.

Therefore there exists at least one point $w_0\in G$ with the property $A_{0,r_0-1}(w_0)\neq 0$. Put, as above, $F:=\{w\in G: A_{0,r_0-1}(w)=0\}$ and note that $F$ is a proper analytic set in $G$.

By assumption, if $w\in G\setminus F$, then the following system of equations has a solution
$$
\left(\begin{matrix} \wdht c_0(w) & \cdots & \wdht c_{r_0-1}(w)\\
&\cdots &\\
\wdht c_{r_0-1}(w) & \dots & \wdht c_{2r_0-2}(w) \end{matrix}\right)
\left(\begin{matrix} \wdht b_{r_0}(w)\\ \vdots \\ \wdht b_{1}(w)\end{matrix}\right)
=
\left(\begin{matrix} \wdht c_{r_0}(w)\\ \vdots \\ \wdht c_{2r_0-1}(w)\end{matrix}\right).
$$

We get that the $\wdht b_0:=-1,\dots, \wdht b_{r_0}$ solve all the equations
$$
\wdht c_{s+0}(w)\wdht b_{r_0}(w)+ \cdots + \wdht c_{s+r_0}(w)\wdht b_0(w)=0 , \quad w\in G\setminus F,\; s=0,\dots r_0-1.
$$
Applying Cramer's rule, we see that $\wdht b_j(w)=\wdtl b_j(w)/A_{0,r_0-1}(w)$, $j=1,\dots,r_0$, and $\wdtl b_0(w)=-A_{0,r_0-1}(w)$, where these $\wdtl b_j$'s are holomorphic functions on $G$. Hence,
$$
\wdht c_{0+s}(w) \wdtl b_{r_0}(w)+ \cdots + \wdht c_{s+r_0}(w) \wdtl b_0(w)=0,\quad w\in G\setminus F,\; s=0,\dots, r_0-1.\quad (\ast)
$$
Recall that $F$ is an analytic set. Therefore, the equations ($\ast$) are also true on the whole of $G$.

Recall the equations ($\ast$) and extend them by an additional one, namely
$$
A_{-1,r_0-1}(w)A_{1,r_0-1}(w)-A_{0,r_0-1}^2(w)=A_{-1,r_0}(w)A_{1,r_0-2}(w),
$$
where $A_{-1,m}(w)$ denotes the determinant of the following Hankel matrix:
$$
H_{-1,m}(w):=\left(\begin{matrix}  1 & \wdht c_{0}(w) & \cdots & \wdht c_{-1+m}(w)\\
\wdht c_{0}(w) & \wdht c_{1}(w) & \cdots & \wdht c_{m}(w)\\
\cdots\\
\wdht c_{-1+m}(w) & \wdht c_{m}(w) & \cdots & \wdht c_{-1+2m}(w)\end{matrix}\right).
$$

Note that the right hand side always vanishes. Since $A_{0,r_0-1}(w)\neq 0$, it follows that then $A_{1,r_0-1}(w)\neq 0$, and then $A_{2,r_0-1}(w)\neq 0$ etc., i.e. $A_{s,r_0-1}(w)\neq 0$ for all $s\geq 0$.

Then the last row of the matrix $H_{s,r_0}(w)$ is always linear dependent of its first rows. Therefore, we get step by step that all the equations
$$
\wdht c_{0+s}(w) \wdtl b_{r_0}(w)+ \cdots + \wdht c_{s+r_0}(w) \wdtl b_0(w)=0, \quad w\in G,\; s\geq 0 .\quad (\ast \ast)
$$
are correct.

Finally, put for $w\in G$:
$$\begin{aligned}
\wdtl a_0(w)&:= \wdtl b_0(w)\wdht c_0(w),\\
\wdtl a_1(w)&:= \wdtl b_1(w)\wdht c_0(w)+ \wdtl b_0(w)\wdht c_1(w),\\
\cdots\\
\wdtl a_{r_0-1}(w)&:= \wdtl b_{r_0-1}(w)\wdht c_0(w) + \wdtl b_{r_0-2}\wdht c_1(w)+\cdots + \wdtl b_0(w)\wdht c_{r_0-1}(w).
\end{aligned}
$$
Obviously, all the  $\wdtl a_j\in\CO(G)$.
Put
$$
\begin{aligned}
P(z,w)&:=P(w)(z):=\sum_{j=0}^{r_0-1} \wdtl a_j(w)z^j, w\in G,\\
Q(z,w)&:=Q(w)(z):=\sum_{j=0}^{r_0}\wdtl b_j(w)z^j, w\in G.
\end{aligned}
$$
Then we obtain
$$
\wdht f(z,w)Q(w)(z)=P(w)(z),\quad w\in G, \; |z|<\eps(w).
$$
In particular, $f(z,w)Q(z,w)=P(z,w)$ if $w\in B$, $|z|<\eps(w)$. Using the identity theorem one has even more, namely, $fQ=P$ on $D\times B$.
Moreover, if $w\in B\setminus F$, then $Q(\cdot,w)$ is a non trivial polynomial. Therefore, if $w\in B\setminus F$, then $f(\cdot,w)$ is a rational function with coefficients holomorphic on $G$.

Hence the case  $m=1$ has been proved.
\vskip 0.2cm
For further use we emphasize that {\it all the coefficients of $P$ and $Q$ are given as certain polynomials in finitely coefficients of the power series of $\wdht f(\cdot,w)$ } via Cramer's rule.
\vskip 0.4cm

Now, let $m\geq 2$ and assume that Theorem \ref{zwei} is true for the case $m-1$. Again, we may  assume that $0\in A$ and that$f(\cdot,w)\not\equiv 0$ for some $w\in B$..

By assumption, for any $w\in B$ there exist polynomials $P_w, Q_w\in\CC[z_1,\dots,z_m]$, $Q_w$ not identically zero, such that for all $z\in D$ the following representation $f(z,w)Q_w(z)=P_w(z)$ holds.
Moreover, we may assume that both polynomials are relatively prime. Following \cite{RudSto}, then $Q_w(z) \neq 0$ for all $z\in D$.

Let $z'\in\CC^{m-1}$ such that $D_{z'}:=\{z_1\in\CC:(z_1,z')\in D\}\neq \varnothing$. Note that, in general, $D_{z'}$ is not connected. Then, for $w\in B$, we have  $f(\cdot,z',w)Q_w(\cdot,z')=P_w(\cdot,z')$ on $D_{z'}$.
Note that $P_w(\cdot,z')$ and $Q_w(\cdot,z')$ are polynomials in $z_1$ and $Q_w(\cdot,z')$ is nowhere zero on $D_{z'}$.

The next step consists in defining the corresponding cross for $\wdht f$ to apply the situation for $m=1$:

Fix a relatively compact subdomain $G'$ of $G$ which contains $B$. Obviously, $G'$ may be chosen as a domain of holomorphy. Then there exists $M<1$ such that $h^\ast_{B,G}<M$ on $G'$.
Put $D_M:=\{z\in D: h^\ast_{A,D}(z)<1-M\}$. Note that $A\subset D_M$ and $D_M\times G'\subset D\times G$.

Note that $D_M$ is open. Thus there exists a positive number $\rho$ such that $\PP_m(\rho)\subset D_M$, where $\PP_{m}(\rho):=\{z\in\CC^{m}:\|z\|<\rho\}=\PP_1(\rho)\times\PP_{m-1}(\rho)$ is the polycylinder with
center $0$ and radius $\rho$. Choose $\rho'<\rho$ near $\rho$.  Put
$$
\bold Y:=\big(\PP_1(\rho)\times(\PP_{m-1}(\rho)\times G')\big)\cup \big(\PP_1(\rho)\times(\overline{\PP_{m-1}(\rho')}\times B)\big).
$$
Then $\wdht f|_{\bold Y} \in\CO_s(\bold Y)$ and for any point $(z',w)\in \overline{\PP_{m-1}(\rho')}\times B$ there are polynomials $P_w(\cdot,z')$ and $Q_w(\cdot,z')$ in $z_1$, the last without zeros on $\PP_1(\rho)$ such that
$$
\wdht f(\cdot,z',w)Q_w(\cdot,z')=P_w(\cdot,z') \quad\text{on}\quad \PP_1(\rho).
 $$
 Two remarks are in order, namely that $\PP_1(\rho)$ is locally pluriregular in $\PP_1(\rho)$ and $\overline{\PP_{m-1}(\rho')}\times B$ is compact and locally pluriregular in $\PP_{m-1}(\rho)\times G$.

Summarizing: $\wdht f$ and $\bold Y$ satisfy the assumption for the case $m=1$. Therefore, there exist polynomials $P, Q\in\CO(\PP_{m-1}(\rho)\times G')[z_1]$,  such that
$$
\wdht f Q= P \quad\text{ on }\quad \bold Y. \quad (\ast\ast\ast)
$$
Moreover, there exists a proper analytic subset $F\subset \PP_{m-1}(\rho)\times G'$ such that $Q(\cdot,z',w)$ is not the zero polynomial in $z_1$, as long as $(z',w)\in \big(\overline{ \PP_{m-1}(\rho')}\times B\big)\setminus F$.

Let $P$ and $Q$ be given as
$$
P(z',w)(z_1):=\sum_{j=0}^{r-1}\wdtl a_j(z',w)z_1^j\quad \text{and }\quad Q(z',w)(z_1):=\sum_{j=0}^{r}\wdtl b_j(z',w)z_1^j
$$
 for a certain $r\in\NN$. Recall that the coefficients $\wdtl a_j$ and $\wdtl b_j$ are holomorphic functions on
$\PP_{m-1}(\rho)\times G'$. In case that $(z',w)\in\big(\PP_{m-1}(\rho)\times G'\big)\setminus F$, then there exists an index $j_{z',w}\in\{0,\dots,r\}$ such that $b_{j_{z',w}}(z',w)\neq 0$.

Moreover, these coefficients are polynomials in finitely many of the coefficients of the power series in $z_1$ of $\wdht f(\cdot,z',w)$, i.e. in terms as $c_j(z',w):=\frac{1}{j!}\frac{\partial^j f}{\partial z_1^j}(0,z',w)$.

Observe that the $c_j$ are holomorphic in $\PP_{m-1}(\rho)\times G'$.

Note that by assumption $f(\cdot,w)$ is a rational function in $z$ (recall that $f(z,w)Q_w(z)=P_w(z)$ and that $Q_w$ is without zeros on $\PP_m(\rho)$). Taking the $z_1$-derivatives of $f(z_1,z',w)=P_w(z_1,z')/Q_w(z_1,z')$ at $z_1=0$ gives that $c_j(z',w)=P_{w,j}(z')/Q_{w,j}(z')$ with polynomials $P_{w,j}, Q_{w,j}\in\CC[z']$ and $Q_{w,j}$ without zeros on $\PP_{m-1}(\rho)$.

Then the $N$ coefficients $c_j$ , $j=0,\dots,N$ together with the cross $(\PP_{m-1}(\rho)\times G')\cup(\PP_{m-1}(\rho)\times B)$ satisfy the induction assumption for $m-1$. Thus there exist $N$ domains of holomorphy $G_{j+1}'\subset G_{j}'\subset G'$ with $B\subset G_N'$, proper analytic sets $F_j\subset G_j'$, polynomials $p_j, q_j\in \CO(G_j')[z']$ with the following properties:
$$
\begin{aligned}
c_j(z',w)&q_j(z',w)=p_j(z',w) \text{ if }w\in B, \\
&q_j(\cdot,w) \quad\text{ is not the trivial polynomial if } w\in G_j'\setminus F_j.
\end{aligned}
$$

Put $G'':=G_N'$ and $\wdht F_1:=\bigcup_{s=0}^N (F_j\cap G'')$. Then $\wdht F_1$ is a proper analytic subset of $G''$.

Recall that $G''$ is a domain of holomorphy. Then the analytic set $F\subset\PP_{m-1}(\rho)\times G''$ is given as the simultaneous zero set of a family of holomorphic functions $h_j\in\CO(\PP_{m-1}(\rho)\times G'')$, $j\in J$
(see \cite{Gun-Ros 1965}, Chap. VIII, 18.Theorem). Put $\wdht F_2:=\{w\in G'': h_j(z,w)=0 \text{ for all }z\in \PP_{m-1}(\rho), j\in J\}$. Then $\wdht F_2$ is a proper analytic subset of $G''$ (see \cite{Whit1972}, Theorem 9E).
Finally, put $\wdht F:=\wdht F_1\cup \wdht F_2$. Then $\wdht F$ is a proper analytic subset of $G''$.

Take now a point $w_0\in B\setminus \wdht F$. Then for any $j=1,\dots,N$ there exists an proper analytic subset $E_j\subset \CC^{m-1}$ such that $q_j(\cdot,w_0)$ has no zeros on $\CC^{m-1}\setminus E_j$.
Put $\wdht E:=\cup_{j=1}^N E_j$. Then $q_j(z',w_0)\neq 0$ whenever $z'\in\CC^{m-1}\setminus\wdht E$ and $j=1,\dots,N$.

Observe that $\wdht E$ is a proper analytic subset on $\CC^{m-1}$. In particular, it is closed without inner points.

Then there are two possibilities: either there is a point $z'_0\in \overline\PP_{m-1}(\rho')\setminus \wdht E$ such that $(z'_0,w_0)\in (\overline\PP_{m-1}(\rho')\times B)\setminus F$ or
$(\overline\PP_{m-1}(\rho')\setminus \wdht E)\times\{w_0\}\subset F$.

In the first case we know that  $\wdtl b_{j_0}(z'_0,w_0)\neq 0$ for some $j_0\in\{1,\dots,N\}$ and $q_s(z'_0.w_0)\neq 0$ for all $s=0,\dots,N$.

Recall that $\wdtl b_{j}$'s (resp. $\wdtl a_{j}$'s) are given as  polynomials in the $p_k/q_k$'s, i.e.
$$
\begin{aligned}
\wdtl b_{j}(z',w)=&\sum_{k=0}^{N_{j}} y_{j,k}\Big(\prod_{s=0}^N\big(\frac{p_s(z',w)}{q_s(z',w)}\big)^{n_{j,k,s}}\Big), \\
&\wdtl a_{j}(z',w)=\sum_{k=0}^{M_{j}} x_{j,k}\Big(\prod_{s=0}^N\big(\frac{p_s(z',w)}{q_s(z',w)}\big)^{m_{j,k,s}}\Big),
\end{aligned}
$$
where the $x_{j,k}$, $y_{j,k}$ are complex numbers and $m_{k,s}, n_{j,k,s}\in\NN_0$ and $(z',w)\in\PP_{m-1}\times G''$.

Inserting this into the polynomials $P, Q$ we  get
$$
\begin{aligned}
Q(z',w)(z_1)=\sum_{j=0}^r\Big(\sum_{k=1}^{N_j} y_{j,k}\Big(\prod_{s=0}^N\big(\frac{p_s(z',w)}{q_s(z',w)}\big)^{n_{j,k,s}} \Big)z_1^j,\\
P(z',w)(z_1)=\sum_{j=0}^{r-1}\Big(\sum_{k=1}^{M_j} x_{j,k}\Big(\prod_{s=0}^N\big(\frac{p_s(z',w)}{q_s(z',w)}\big)^{m_{j,k,s}} \Big)z_1^j.
\end{aligned}
$$
Multiplying the equation $(\ast\ast\ast)$ with $\prod_{s=0}^N q_s(z',w)^{K_s}$, where $\wdht N_s:=\max\{n_{j,k,s}:j=0,\dots,r, k=0,\dots, M_j\}$, $\wdht M_s:=\max\{m_{j,k,s}:j=0,\dots,r-1, k=0,\dots, N_j\}$,and $K_s:= \max\{\wdht N_s,\wdht M_s\}$, yields

$$
\begin{aligned}
f(z_1,z',w) \sum_{j=0}^r\Big(\sum_{k=1}^{N_j}& y_{j,k}\big(\prod_{s=0}^N p_s(z',w)^{n_{j,k,s}} q_s(z',w)^{K_s-n_{j,k,s}}\big)\Big)z_1^j \\
&=\sum_{j=0}^{r-1}\Big(\sum_{k=1}^{M_j} x_{j,k}\big(\prod_{s=0}^N p_k(z',w)^{m_{j,k}}q_k(z',w)^{K_s-m_{{j,k,s}}}\big) \Big)z_1^j.
\end{aligned}
$$
Therefore, one has the following equation for $f$, namely  $f(z_1,z',w)\wdht Q(z_1,z',w)=\wdht P(z_1,z',w)$, where $\wdht Q(z_1,z',w), \wdht P(z_1,z',w)\in\CO(G')[z_1,z']$.

Recall that
$$
\begin{aligned}
0\neq \wdtl b_{j_0}(z'_0,w_0)&\prod_{s=0}^N q_s(z'_0,w_0)^{K_s} =\\ &\sum_{k=1}^{N_{j_0}} y_{j_0,k}\big(\prod_{s=0}^N p_s(z'_0,w_0)^{n_{j_0,k,s}} q_s(z'_0,w_0)^{K_s-n_{j_0,k,s}}\big).
\end{aligned}
$$
Hence $\wdht Q(\cdot,w_0)$ is not the zero polynomial or $f(\cdot,w_0)$ is the restriction of a rational function whenever $w_0\in B\setminus \wdht F$. So the first case is completely done.

Recall the second case which is left so far, i.e. from now on we assume
$$
w_0\in B\setminus\wdht F \quad\text{ and }\quad(\overline\PP_{m-1}(\rho')\setminus \wdht E)\times\{w_0\}\subset F.
$$
Then even more is true, namely $\overline\PP_{m-1}(\rho')\times\{w_0\}\subset F$. And then $\PP_{m-1}(\rho)\times\{w_0\}\subset F$ (use the identity theorem for the functions $h_j(\cdot,w_0)$); a contradiction to the assumption that $w_0\notin \wdht F_2$.

Hence the theorem is proven.
\end{proof}

{\bf Acknowledgement.} I deeply thank Sergej Melikhov for answering my questions connected with his paper which finally enabled me to write this note. Also I like to thank Arkadiusz Lewandowski for  careful readings of some of the versions of this note.

\bibliographystyle{amsplain}

\end{document}